\documentclass[12pt]{amsart}
\usepackage{amsmath}
\usepackage{amsfonts}
\usepackage{amssymb}
\usepackage{amsthm}
\usepackage{amscd}
\usepackage{tikz}
\usepackage{dsfont}
\usetikzlibrary{matrix}
\usepackage{enumitem}
\newcommand{\CP}{\mathbb{CP}}
\newcommand{\la}{\langle}
\newcommand{\ra}{\rangle}
\newcommand{\x}{\times}

\newcommand{\lrightarrow}{\longrightarrow}

\newcommand{\SC}{{\mathcal{C}}}

\newcommand{\ZZ}{\mathbb{Z}}
\newcommand{\CC}{\mathbb{C}}
\newcommand{\RR}{\mathbb{R}}

\DeclareMathOperator{\Id}{Id}

\numberwithin{equation}{section}

\theoremstyle{plain}
\newtheorem{proposition}{Proposition}[section]
\newtheorem{theorem}[proposition]{Theorem}

\theoremstyle{definition}
\newtheorem{definition}[proposition]{Definition}

\theoremstyle{remark}
\newtheorem{remark}[proposition]{Remark}

\usepackage[T1]{fontenc}
\usepackage[cp1250]{inputenc}
\usepackage{amssymb}
\usepackage{amsfonts}
\usepackage{multirow}
\usepackage{multicol}
\usepackage[a4paper,margin=2cm]{geometry}

\newcommand{\PP}{\mathbb{P}}

\newcommand{\SU}{\operatorname{SU}}

\newcommand{\bt}{\mathbf{t}}

\newcommand{\ox}{\otimes}

\newcommand{\orb}{\mathrm{orb}}

\theoremstyle{plain}
\theoremstyle{definition}
\newtheorem{example}[proposition]{Example}
\theoremstyle{remark}
\newtheorem{question}{Question}
\newenvironment{dedication}
    {\vspace{6ex}\begin{quotation}\begin{center}\begin{em}}
    {\par\end{em}\end{center}\end{quotation}}

\title{Some topics in Sasakian geometry, a survey}
\author[A. Tralle]{Aleksy Tralle}
\address{Faculty of Mathematics and Computer Science,\\
University of Warmia and Mazury\\
 S\l\/oneczna 54, 10-710 Olsztyn, Poland}
\email{tralle@matman.uwm.edu.pl}


\begin{document}
\maketitle

\begin{dedication}
To Vladimir Rovenski's 70-th birthday
\end{dedication}

\begin{abstract} In the seminal book of Boyer and Galicki {\it Sasakian Geometry} the authors formulated a research program of studying topological properties and answering questions about the  existence of Sasakian structures. We survey recent progress in this topic.
\end{abstract}

\section{Introduction and preliminaries}\label{sec:intro}
 Sasakian 
geometry has become an important and active subject, especially after the appearance of 
the fundamental treatise of Boyer and Galicki \cite{BG}. The Chapter 7 of this book 
contains an extended discussion on the topological problems in the theory of Sasakian, 
and, more generally, $K$-contact manifolds. This article is a survey based on recent results \cite{BBMT, BFMT, CMST, CMRV, HT, M,M1, MRT, MT, MT1, MST1}. The aim of this survey is to attract the attention of researchers to the theory of Seifert bundles over cyclic symplectic or K\"ahler orbifolds which may be of importance for the development of Sasakian geometry. The foundations of this theory were laid by Koll\'ar \cite{K,K1}. However, in its original version, it imposed  restrictions of smoothness on the base cyclic orbifolds. Papers \cite{CMST,M,M1, MST1} contain further development of Koll\'ar's theory which includes non-smooth orbifolds. This enabled the authors to solve several open problems posed in \cite{BG}. Also, we include into this survey several results on rational homotopy properties of Sasakian manifolds \cite{BBMT, BFMT,CNY,CNMY, HT, MT1} and on the classification of Smale-Barden manifolds with Sasakian strcutures \cite{MT}. These results answer questions posed in \cite{BG}. It should be mentioned that  results on rational homotopy properties of Sasakian manifolds use  Tievsky's model (in the sense of rational homotopy theory) \cite{Tievsky}. This construction is contained in the Ph. D. thesis of Tievsky and may be not well known. In this survey we attract the attention of the reader to the use of this model, which may be of importance in establishing new topological properties of Sasakian manifolds. It should be mentioned that although we concentrate on the particluar articles mentioned above, they are inspired and motivated by results of Boyer, Galicki  and their co-authors \cite{BG,BGK, BGM, BGM1, BGN, BN}.  Also, it would be interesting to think over a more general approach to "Sasaki-like" structures developed by Patra and Rovenski in \cite{PV} in the framework of this article.

Let us mention that we freely use the basics of symplectic topology \cite{McD}, the theory of compact complex surfaces \cite{BPV}, complex and symplectic surgery \cite{GS} and some facts from algebraic geometry \cite{H}. All the required material is contained in these sources, but we do not discuss it.   

Finally, note that  we review known results and proofs.

The structure of this survey is as follows.
\begin{enumerate}
\item In Section \ref{sec:intro} we define Sasakian and K-contact manifolds and formulate Rukimbira's theorem \cite{R} which enables us to consider these manifolds as total spaces of Seifert bundles over singular cyclic orbifolds. Also, we discuss some basic facts of rational homotopy theory used in Section \ref{sec:sas-top} stressing Tievsky's algebraic  models of Sasakian manifolds.
\item In Section \ref{sec:seifert} we discuss methods of construction of singular cyclic orbifolds and Seifert bundles. This is a core section.
\item Section \ref{sec:Smale-Barden} describes  applications of the theory developed in Section \ref{sec:seifert} to several existence problems posed in \cite{BG}. 
\item Section \ref{sec:sas-top} is devoted to rational homotopy properties of Sasakian manifolds. In particular, it is shown that triple Massey products do not obstruct Sasakian structures in the simply connected case, but higher order Massey products do. This leads to a solution of the problem posed in \cite{BG} about the existence of simply-connected K-contact manifolds with no Sasakian structures.
\end{enumerate}
Finally, let us note that each of the questions from \cite{BG} is discussed separately in Sections \ref{sec:Smale-Barden} and \ref{sec:sas-top}.

Basically, we omit the proofs of the described results providing the appropriate references. However, in order to show the method of Seifert fibrations,  the use of Tievsky's model as well as symplectic surgery and Donaldson's hyperplane theorem, we show parts of the proofs of Theorem \ref{thm:neg-pos-main}, Theorem \ref{prop:higher-massey}, and  Theorem \ref{thm:main}. 
\subsection{Sasakian and K-contact manifolds}
Let $M$ be a $(2n + 1)$-dimensional manifold. An
{\em almost contact metric structure} on $M$ consists of
a quadruplet $(\eta\, , \xi\, , \phi\, ,g)$, where $\eta$ is a differential
$1$-form, $\xi$ is a nowhere vanishing vector field 
(the {\em Reeb} vector field), $\phi$ is a
$\SC^\infty$ section of ${\mathrm End}(TM)$ and $g$ is a Riemannian metric on $M$,
satisfying the following conditions
\begin{equation}\label{almostcontactmetric}
\eta(\xi) \,=\, 1, \quad \phi^2\,=\, - \Id + \xi \otimes \eta, \quad
g (\phi X\, , \phi Y)\,= \,g(X\, , Y) - \eta(X) \eta(Y)\, ,
\end{equation}
for all vector fields $X, Y$ on $M$. 
Thus, the kernel of $\eta$ defines a
codimension one distribution ${\mathcal D} \,=\, \ker (\eta)$, and there is
the orthogonal decomposition of the tangent bundle $TM$ of $M$
$$
TM = {\mathcal D} \oplus {\mathcal L}\, ,
$$
where ${\mathcal L}$ is the trivial line subbundle of $TM$ generated by $\xi$. Note
that conditions in \eqref{almostcontactmetric} imply that
\begin{equation} \label{alphaphi}
\phi(\xi)\,=\, 0\, ,\quad \eta\circ\phi\,=\,0\,.
\end{equation}

For an almost contact metric structure $(\eta\, , \xi\, , \phi\, ,g)$ on $M$, 
the fundamental 2-form $F$ on $M$ is defined by
$$
F (X, Y ) \,=\, g(\phi X, Y )\, ,
$$
where $X$ and $Y$ are vector fields on $M$. Hence, 
$$F(\phi X, \phi Y )\,=\,F(X, Y)\, ,$$ that is $F$ is compatible with
$\phi$, and $\eta\wedge F^n\not=\,0$ everywhere. 

An almost contact metric structure $(\eta\, , \xi\, , \phi\, ,g)$ on $M$ is said to
be a {\em contact metric} if
$$
g (\phi X, Y) \,=\, {d} \eta (X, Y)\,.
$$
In this case $\eta$ is a {\em contact form}, meaning
$$\eta\wedge ({d} \eta)^n\not=\,0$$ at every point of $M$. 
If $(\eta\, , \xi\, , \phi\, ,g)$ is a contact metric structure 
such that $\xi$ is a Killing vector field for $g$, 
then $(\eta\, , \xi\, , \phi\, ,g)$ is called a
{\it K-contact} structure. A manifold with a K-contact structure is called a
{\it K-contact manifold}.

There is the notion of
integrability of an almost contact metric structure. More precisely, an almost contact
metric structure $(\eta\, , \xi\, , \phi\, , g)$ is called \emph{normal} if the Nijenhuis
tensor $N_{\phi}$ associated to the tensor field $\phi$, defined by
\begin{equation}\label{NPhi}
N_{\phi} (X, Y) \,:=\, {\phi}^2 [X, Y] + [\phi X, \phi Y] - 
\phi [ \phi X, Y] - \phi [X, \phi Y]\, ,
\end{equation}
satisfies the equation
$$
N_{\phi} \,=\,-{d}\eta\otimes \xi\, .
$$
This last equation is equivalent to the condition that the almost complex
structure $J$ on $M \times {\mathbb{R}}$ given by
\begin{equation} \label{complexproduct}
J \left( X,\, f\frac{\partial}{\partial t} \right)
\,=\, \left(\phi X - f \xi,\, \eta (X) \frac {\partial} {\partial t} \right)
\end{equation}
is integrable, where $f$ is a smooth function on $M \times {\mathbb{R}}$ and 
$t$ is the standard coordinate on ${\mathbb{R}}$ (see \cite{SH}). In other words, $\phi$
defines a complex structure on the kernel $\ker (\eta)$ compatible with $d\eta$.

A {\it Sasakian structure} is a normal contact metric structure, in other words, an 
almost contact metric structure $(\eta\, , \xi\, , \phi\, , g)$ such that
$$
N_ {\phi} \,=\, -{d} \eta \otimes \xi\, , \quad {d} \eta \,=\, F\, .
$$
If $(\eta\, , \xi\, , \phi\, ,g)$ is a Sasakian structure on $M$, 
then $(M\, ,\eta\, , \xi\, , \phi\, ,g)$ is called a {\em Sasakian manifold}.

Riemannian manifolds with a Sasakian structure can also 
be characterized in terms of the Riemannian cone
over the manifold. A Riemannian manifold $(M,\, g)$ admits 
a compatible Sasakian structure if and only if $M\times {\mathbb{R}^+}$ 
equipped with the cone metric $h\,=\,t^2g+{d}t\otimes{d}t$ is K\"ahler \cite{BG}. 
Furthermore, in this case the Reeb vector field is Killing, and the covariant 
derivative of $\phi$ with respect to the Levi-Civita connection of $g$ is 
given by
\[ (\nabla_X \phi)(Y)\,=\,g(\xi, Y)X-g(X,Y)\xi\, , \]
where $X$ and $Y$ are vector fields on $M$.

In the study of Sasakian manifolds, the {\em basic cohomology} plays an important role. 
Let $(M,\eta\,, \xi\,, \phi\,,g)$ be a Sasakian manifold of dimension
$2n+1$. A differential form $\alpha$ on $M$ is called {\em basic} if
$$
\iota_{\xi}\alpha\,=\,0\,=\,\iota_{\xi}d\alpha\, ,
$$
where $\iota_{\xi}$ denotes the contraction of differential forms by $\xi$. We denote by 
$\Omega_{B}^{k}(M)$ the space of all basic $k$-forms on $M$. Clearly, the de Rham exterior 
differential $d$ of $M$ takes basic forms to basic forms. Denote by $d_B$ the restriction
of the de Rham differential $d$ to $\Omega_{B}^{*}(M)$. The cohomology of the 
differential complex $(\Omega_{B}^{*}(M),\, d_B)$ is called the {\em basic cohomology}.
These basic cohomology groups are denoted by $H_{B}^{*}(M)$. When $M$ is compact, the dimensions of 
the de Rham cohomology groups $H^{*}(M)$ and the basic cohomology groups $H_{B}^{*}(M)$
are related as follows:

\begin{theorem}[{\cite[Theorem 7.4.14]{BG}}]\label{rel:betti-basic-numbers}
Let $(M,\eta\, , \xi\, , \phi\, ,g)$ be a compact 
Sasakian manifold of dimension $2n+1$. Then, the Betti number $b_{r}(M)$ and the basic 
Betti number $b_{r}^{B}(M)$ are related by
$$
b_{r}^{B}(M)\, =\, b_{r-2}^{B}(M)+b_{r}(M)\, ,
$$
for $0\,\leq\, r\,\leq\, n$, In particular, if $r$ is odd
and $r\,\leq\, n$, then $b_{r}(M)\,=\,b_{r}^{B}(M)$.
\end{theorem}

A Sasakian structure on $M$ is called \textit{quasi-regular} if there is a positive
integer $\delta$ satisfying the condition
that each point of $M$ has a foliated coordinate chart
$(U\, ,t)$ with respect to $\xi$ (the coordinate $t$ is in the direction of $\xi$)
such that each leaf for $\xi$ passes through $U$ at most $\delta$ times. If
$\delta\,=\, 1$, then the Sasakian structure is called \textit{regular}  (see
\cite[p. 188]{BG}.)

If $X$ is a compact K\"ahler manifold whose K\"ahler form $\omega$ 
defines an integral cohomology class, then the total space of the circle bundle 
\begin{equation}\label{epi}
S^1\hookrightarrow M\rightarrow X
\end{equation}
with Chern class $[\omega]\,\in\, H^2(M,\,\mathbb{Z})$ is a regular Sasakian manifold with 
a contact form $\eta$ that satisfies the equation $d \eta \,=\, \pi^*(\omega)$, where 
$\pi$ is the projection in \eqref{epi}. This description is known and can be derived from \cite{Kob}.

In the same way, one defines the {\it regularity} and {\it quasi-regularity} of K-contact manifolds. Let $(X,\omega)$ be a symplectic manifold such that the cohomology class $[\omega]$ is integral. Consider the principal $S^1$-bundle $\pi:M\rightarrow X$ given by the cohomology class $[\omega]\in H^2(X,\mathbb{Z})$. Such fiber bundles are called {\it Boothby-Wang fibrations}. We have the following result (which is also known, and can be derived, for example, from \cite{L2}). 
\begin{theorem} Any Boothby-Wang fibration admits a K-contact structure on the total space.
\end{theorem}
It is important to note the following result of Rukimbira. 
\begin{proposition}[\cite{R}] \label{prop:quasi-regular}
If a compact manifold $M$ admits a K-contact structure, it admits a quasi-regular K-contact structure. If a compact manifold $M$ admits a Sasakian structure, it admits a quasi-regular Sasakian structure.
\end{proposition}
It follows that analyzing topological properties of K-contact/Sasakian manifolds, we can always assume that we are given {\it quasi-regular} structures. It follows that our approach {\it  is based on constructing Seifert bundles over K\"ahler or symplectic cyclic orbifolds (or manifolds, for regular structures)}. 
In some cases, it is useful to have an intermediate class of {\it semi-regular} Sasakian structures \cite{MRT}. A cyclic  orbifold is {\it smooth}, if all its points are smooth (see Proposition 2 and Definition 3 in \cite{MRT}). Seifert bundles over smooth orbifold yield semi-regular Sasakian strcuture. We do not present all the details here, because we are interested mainly in the general case.

\subsection{Formal manifolds and Massey products}\label{formal manifolds}

In this section some definitions and results about minimal models
and Massey products are reviewed \cite{DGMS, FT, TO}. These will be used in Section \ref{sec:sas-top}.

We work with {\em differential graded commutative algebras}, or DGAs,
over the field $\mathbb R$ of real numbers. The degree of an
element $a$ of a DGA is denoted by $|a|$. A DGA $(\mathcal{A},\,d)$ is {\it minimal\/}
if:
\begin{enumerate}
 \item $\mathcal{A}$ is free as an algebra, that is, $\mathcal{A}$ is the free
 algebra $\bigwedge V$ over a graded vector space $V\,=\,\bigoplus_i V^i$, and

 \item there is a collection of generators $\{a_\tau\}_{\tau\in I}$
indexed by some well ordered set $I$, such that
 $|a_\mu|\,\leq\, |a_\tau|$ if $\mu \,< \,\tau$ and each $d
 a_\tau$ is expressed in terms of preceding $a_\mu$, $\mu\,<\,\tau$.
 This implies that $da_\tau$ does not have a linear part.
 \end{enumerate}

In 
our context, the main example of DGA is the de Rham complex $(\Omega^*(M),\,d)$
of a differentiable manifold $M$, where $d$ is the exterior differential.

Given a differential graded commutative algebra $(\mathcal{A},\,d)$, we 
denote its cohomology by $H^*(\mathcal{A})$. The cohomology of a 
differential graded algebra $H^*(\mathcal{A})$ is naturally a DGA with the 
product structure inherited from that on $\mathcal{A}$ and with the differential
being identically zero. 
According to \cite[page 249]{DGMS}, an {\it elementary extension} 
of a differential graded commutative algebra $(\mathcal{A},\,d)$ is any 
DGA of the form 
$$(\mathcal{B}=\mathcal{A}\otimes\bigwedge V,\, d_{\mathcal{B}})\, ,$$ 
where 
$d_{\mathcal{B}}{|}_{\mathcal{A}}\,=\,d$, $d_{\mathcal{B}}(V)
\,\subset\,\mathcal{A}$ and $\bigwedge V$
is the free algebra over a finite dimensional vector space $V$ whose elements have all the same degree.

Morphisms between DGAs are required to preserve the degree and to commute 
with the differential. 

We shall say that $(\bigwedge V,\,d)$ is a {\it minimal model} of the
differential graded commutative algebra $(\mathcal{A},\,d)$ if $(\bigwedge V,\,d)$ is minimal and there
exists a morphism of differential graded algebras $$\rho\,\colon\,
{(\bigwedge V,\,d)}\longrightarrow {(\mathcal{A},\,d)}$$ inducing an isomorphism
$\rho^*\colon H^*(\bigwedge V)\stackrel{\sim}{\longrightarrow}
H^*(\mathcal{A})$ of cohomologies.

In~\cite{Halperin}, Halperin proved that any connected differential graded algebra
$(\mathcal{A},\,d)$ has a minimal model unique up to isomorphism.

A {\it minimal model\/} of a connected differentiable manifold $M$
is a minimal model $(\bigwedge V,\,d)$ for the de Rham complex
$(\Omega^*(M),\,d)$ of differential forms on $M$. The importance of the minimal model comes from the fact that it is a {\it homotopy invariant} \cite{FT,TO}.

Recall that a minimal algebra $(\bigwedge V,\,d)$ is called
{\it formal} if there exists a
morphism of differential algebras $\psi\colon {(\bigwedge V,\,d)}\,\longrightarrow\,
(H^*(\bigwedge V),0)$ inducing the identity map on cohomology.
Also a differentiable manifold $M$ is called {\it formal\/} if its minimal model is
formal. Many examples of formal manifolds are known: spheres, projective
spaces, compact Lie groups, symmetric spaces, flag manifolds,
and all compact K\"ahler manifolds \cite{FOT, TO}.

 {\it Massey
products} are known to be obstructions to formality \cite{DGMS}. 
 Let $(\mathcal{A},\,d)$ be a DGA (in particular, it can be the de Rham complex
of differential forms on a differentiable manifold). Suppose that there are
cohomology classes $[a_i]\,\in\, H^{p_i}(\mathcal{A})$, $p_i\,>\,0$,
$1\,\leq\, i\,\leq\, 3$, such that $a_1\cdot a_2$ and $a_2\cdot a_3$ are
exact. Write $a_1\cdot a_2\,=\,da_{1,2}$ and $a_2\cdot a_3\,=\,da_{2,3}$.
The {\it (triple) Massey product} of the classes $[a_i]$ is defined to be
$$
\langle [a_1],[a_2],[a_3] \rangle \,=\, 
[ a_1 \cdot a_{2,3}+(-1)^{p_{1}+1} a_{1,2}
\cdot a_3]
$$
$$
\in \, \frac{H^{p_{1}+p_{2}+ p_{3} -1}(\mathcal{A})}{[a_1]\cdot 
H^{p_{2}+ p_{3} -1}(\mathcal{A})+[a_3]\cdot H^{p_{1}+ p_{2} -1}(\mathcal{A})}\, .
$$
 
Now we move on to the definition of higher Massey products
(see~\cite{TO}). Given $$[a_{i}]\,\in\, H^*(\mathcal{A})\, ,\ \
1\,\leq\,i\,\leq\, t\, , \ \ t\,\geq\, 3\, ,$$
the Massey product $\la [a_{1}],[a_{2}],\cdots,[a_{t}]\ra$, is defined
if there are elements $a_{i,j}$ on $\mathcal{A}$, with $1\,\leq\, i\,\leq\, j\,\leq\,
t$ and $(i,j)\,\not= \,(1,t)$, such that
 \begin{equation}\label{eqn:gm}
 \begin{aligned}
 a_{i,i}&\,=\, a_i\, ,\\
 d\,a_{i,j}&\,=\, \sum\limits_{k=i}^{j-1} (-1)^{|a_{i,k}|}a_{i,k}\cdot
 a_{k+1,j}\, .
 \end{aligned}
 \end{equation}
Then the {\it Massey product} $\la [a_{1}],[a_{2}],\cdots,[a_{t}] \ra$ is the set of cohomology classes
 $$
 \la [a_{1}],[a_{2}],\cdots,[a_{t}] \ra
$$
$$
=\, \left\{
 \left[\sum\limits_{k=1}^{t-1} (-1)^{|a_{1,k}|}a_{1,k} \cdot
 a_{k+1,t}\right] \ \mid \ a_{i,j} \mbox{ as in (\ref{eqn:gm})}\right\}
 \,\subset\, H^{|a_{1}|+ \ldots +|a_{t}|
 -(t-2)}(\mathcal{A})\, .
 $$
We say that the Massey product is {\it zero} if
$$0\,\in\, \la [a_{1}],[a_{2}],\cdots,[a_{t}] \ra\, .$$ 

It should be mentioned that for $\la
a_{1},a_{2},\cdots,a_{t}\ra$ to be defined, it is necessary that
all the lower order Massey products $\la a_{1},\cdots,a_{i}\ra$ and
$\la a_{i+1},\cdots,a_{t}\ra$ with $2 \,<\, i \,<\, t-2$ are defined and
trivial. 

Massey products are related to formality by the
following well-known result.

\begin{theorem}[\cite{DGMS,TO}] \label{theo:Massey products}
A DGA which has a non-zero Massey product is not formal.
\end{theorem}

\subsection{Tievsky's model}
 Let $(M,\eta\, , \xi\, , \phi\, ,g)$ be a 
compact Sasakian manifold.  Tievsky in \cite{Tievsky}, motivated by the 
approach in \cite{DGMS} and using the basic cohomology of $M$, found a model for 
$M$, that is, a DGA with the same minimal model as the manifold $M$.

\begin{theorem}[\cite{Tievsky}] \label{Tievsky model}
Let $(M,\eta\, , \xi\, , \phi\, ,g)$ be a compact Sasakian manifold. Then a model for $M$ is
given by the DGA $(H_{B}^{*}(M)\otimes\bigwedge(x),\, D)$, where $|x|\,=\,1$,
$D(H_{B}^{*}(M))\,=\,0$ and $Dx\,=\,[d\eta]_{B}$. Therefore,
$(H_{B}^{*}(M)\otimes\bigwedge(x), \,D)$ is an elementary 
extension of the DGA $(H_{B}^{*}(M),\, 0)$ (the differential is zero).
\end{theorem}

\section{Quasi-regular Sasakian manifolds as Seifert bundles over cyclic orbifolds}\label{sec:seifert}
Thus, {\it regular} K-contact structures are Boothby-Wang fibrations. Let us discuss the {\it quasi-regularity condition} for K-contact and Sasakian structures. First, we refer to \cite{BG} for the general theory of (cyclic) orbifolds, and we freely use this theory without further explanations concentrating on {\it constructions} of K\"ahler and symplectic {\it singular} orbifolds.  Quasi-regular Sasakian and K-contact structures arise as {\it Seifert bundles} over such orbifolds. Full details can be found in \cite{M}.

\subsection{Symplectic and K\"ahler singular orbifolds and Seifert bundles}
\begin{definition}[\cite{M}] A cyclic singular symplectic (K\"ahler) orbifold $X$ is a symplectic orbifold whose isotropy set is of dimension 0 (that is, a finite set $P$ of points called singular set).
\end{definition}
For a cyclic singular symplectic 4-manifold, a singular point is an isolated isotropy point $x\in P\subset X$. A local model around $x$ is of the form $\mathbb{C}^2/\mathbb{Z}_d$, where $\xi=\exp(2\pi i/d)$ acts as 
$$\xi\cdot (z_1,z_2)=(\xi^{e_1}z_1,\xi^{e_2}z_2)$$
where $\operatorname{gcd}(e_1,d)=\operatorname{gcd}(e_2,d)=1$. We will write $d=d(x)$.
\begin{definition}\label{def:sing}[\cite{M}] A sing-symplectic surface is a symplectic 2-orbifold $D\subset X$ such that if $x\in D$ is a singular point, then $D$ is fixed by the isotropy subgroup. Two sing-symplectic surfaces $D_1$ and $D_2$ intersect nicely if at at every intersection point $x\in D_1\cap D_2$ there is an adapted Darboux chart $(z_1,z_2)$ centered at $x$ such that $D_1=\{(z_1,0)\}$ and $D_2=\{(0,z_2)\}$ in a model $\mathbb{C}^2/\mathbb{Z}_d$, where $\mathbb{Z}_d$ is a subgroup of $U(1)\times U(1)$.
\end{definition}
The construction of  such orbifolds is based on the following result (note that, for simplicity, we restrict ourselves to dimension 4).
\begin{theorem}[\cite{M}]\label{thm:munoz} Let $X$ be a cyclic singular 4-manifold with the set of singular points $P$. Let $D_i$ be embedded sing-symplectic surfaces intersecting nicely, and take coefficients $m_i>1$ such that $\operatorname{gcd}(m_i,m_j)=1$ if $D_i$ and $D_j$ have a non-empty intersection. Then there exists an orbifold $X$ with isotropy surfaces $D_i$ of multiplicities $m_i$, and singular points $x\in P$ of multiplicity $m=d(x)\prod_{i\in I_x}m_i,I_x\{i\,|\,x\in D_i\}$.
\end{theorem}
Note that if one replaces the word "symplectic" by "K\"ahler" one obtains cyclic singular K\"ahler orbifold. If one takes $P=\emptyset$ one obtains smooth cyclic orbifolds as in \cite{K1}. Finally, it is important to understand that this theorem, indeed, shows how to construct K\"ahler or symplectic orbifolds which are non-smooth: one must construct the data
\begin{equation}
X,P,D_i\subset X, m_i, m=d(x)\prod_{i\in I_x}m_i,I_x=\{i\,|\,x\in D_i\},
\end{equation}\label{eqn:orb}
satisfying the restrictions given by Theorem \ref{thm:munoz}.  
\begin{definition}\label{def:seifert} Let $X$ be a cyclic $n$-orbifold. A Seifert bundle over $X$ is an oriented $(n+1)$--manifold endowed with a smooth $S^1$-action and a continuous map $\pi:M\rightarrow X$ such that for an orbifold chart $(U,\tilde U,\mathbb{Z}_m,\varphi)$ there is a commutative diagram 
$$
\CD
(S^1\times \tilde U)/\mathbb{Z}_m @>{\cong}>>\pi^{-1}(U)\\
@VVV @V{\pi}VV\\
\tilde U/\mathbb{Z}_m @>{\cong}>>U
\endCD
$$
where the action of $\mathbb{Z}_m$ on $S^1$ is by multiplication by $\exp(2\pi i/m)$, and the top diffeomorphism is $S^1$-equivariant.
\end{definition} 
The description of Sasakian manifolds is contained in the following result.

\begin{theorem}[{\cite[Theorems 7.5.1 and 7.5.2]{BG}}]\label{thm:sas-orb} 
Let $M$ be a manifold endowed with a quasi-regular Sasakian structure $(\eta,\xi, \phi,g)$. Then the space of leaves of the foliation $\mathcal{F}_{\xi}$  determined by the Reeb vector field has a natural structure of a cyclic K\"ahler orbifold, and the projection $M\rightarrow X$ is a Seifert bundle. Conversely, if $(X,\omega)$ is a cyclic K\"ahler orbifold and $M$ is the total space of the Seifert bundle
 determined by the class $[\omega]$ of the K\"ahler form, then $M$ admits a quasi-regular Sasakian structure.
\end{theorem}
In the same way, one can describe quasi-regular K-contact manifolds replacing cyclic K\"ahler orbifolds by cyclic symplectic orbifolds (see \cite{M} for the details).

\subsection{Constructions of cyclic K\"ahler orbifolds and Seifert bundles}
An important basic tool of constructing cyclic K\"ahler  orbifolds in this work is by blowing-down complex surfaces along chains of smooth rational curves of negative self-intersection $<-1$
(i.e.\ no $(-1)$-curves). We freely use the theory of complex surfaces referring to \cite{BPV}, \cite{GS}, \cite{H}.

\begin{proposition}[\cite{BPV}] \label{prop:Hirz-Jung}
Consider the action of the cyclic group $\ZZ_m$ on $\CC^2$ given by $(z_1,z_2)\mapsto (\eta z_1,\eta^r z_2)$, where
$\eta=e^{2\pi i /m}$, $0<r<m$ and $\gcd(r,m)=1$.
 Then write a continuous fraction
 $$
  \frac{m}{r}=[b_1,\ldots, b_l]=b_1- \frac{1}{b_2-\frac{1}{b_3- \ldots}}
  $$
The resolution of $\CC^2/\ZZ_m$ has an exceptional divisor formed by a chain of 
smooth rational curves of self-intersection numbers
$-b_1,-b_2,\ldots,-b_l$.
\end{proposition}

\begin{proposition}[{\cite[Lemma 15]{CMST}}] \label{prop:constr-orbi} 
Conversely, let $X$ be a smooth complex surface containing a chain of smooth rational curves $E_1,\ldots, E_l$ of self-intersections
$-b_1,-b_2,\ldots$, $-b_l$, with all $b_i\geq 2$, intersecting transversally (so that $E_i\cap E_{i+1}$ are nodes, $i=1,\ldots, l-1$).
Let $\pi:X\to \bar X$ be the contraction of $E=E_1\cup\ldots \cup E_l$. Then $\bar X$ has a cyclic singularity at $p=\pi(E)$,
with an action given by Proposition \ref{prop:Hirz-Jung}. Moreover, if $D$ is a curve
intersecting transversally a tail of the chain (that is, either $E_1$ or $E_l$ at a non-nodal point),
then the push down curve $\bar D=\pi(D)$ is a sing-complex curve curve in $\bar X$ (in the sense of Definition \ref{def:sing}, with appropriate modifications in the complex case). 
\end{proposition}

 Let us mention the following. For the convenience of references, we will interchangebly follow the terminology and notation of \cite{BG}, \cite{K}, \cite{K1} and \cite{M}. If $X$ is a cyclic orbifold with singular set $P$ and a family of surfaces $D_i$  with multiplicities $m_i$ we will say that we are given a divisor $\mathop{\cup} D_i$ with multiplicities $m_i>1$. The formal sum
  \begin{equation}\label{eqn:Delta}
  \Delta=\sum_i\left(1-{\frac 1m_i}\right)D_i
  \end{equation}
will be called the {\it branch divisor}.

Given a cyclic K\"ahler orbifold $X$, with singular points $P$, and branch divisor (\ref{eqn:Delta}), we need an extra piece of information in order
to determine a Seifert bundle $M\to X$. 
For each point $x\in P$ with multiplicity $m=d m_1m_2$ ,
we have an adapted chart $U\subset \CC^2/\ZZ_m$ with action 
   \begin{equation}\label{eqn:Zm}
  \exp(2\pi i/m)(z_1,z_2)=(e^{2\pi i j_1/m}z_1,e^{2\pi i j_2/m}z_2),
  \end{equation}
so that the local model of the Seifert fibration is as in Definition \ref{def:seifert}. 
We call $j_x=(m,j_1,j_2)$ the {\it local invariants} at $x\in P$. Note that 
$j_1=m_1e_1$ and $j_2=m_2e_2$.
Assume that $D_1=\{(z_1,0)\}$ is one of the isotropy surfaces with multiplicity $m_1$. 
The local invariant of $D_1$ is by definition, $j_{D_1}=(m_1,j_2)$, where $j_2$ is considered modulo $m_1$. 
This also yields a {\it compatibility condition} for the local invariants of singular points and isotropy surfaces.

In order to construct Seifert bundles, we need to assign local invariants to each of the singular points $x\in P$ and each
of the isotropy surfaces $D_i\subset X$, in a compatible way. For this we need to choose for each $D_i$ some $j_i$
with $\gcd(j_i,m_i)=1$. Then we can assign local invariants at the singular points using the following result.

\begin{proposition}[{\cite[Proposition 25]{M}}] \label{prop:25-M}
Suppose that $X$ is a cyclic $4$-orbifold 
such that each singular point $x\in P$ lies in a single isotropy surface $D_i$,
if any.
Take integers $j_i$ with $\gcd(m_i,j_i) = 1$ for each $D_i$.  Then there exist local invariants for $X$.
\end{proposition}

Once compatible local invariants are fixed, a Seifert bundle is determined by its orbifold Chern class.
Given a Seifert bundle $\pi: M\rightarrow X$, the order of a stabilizer (in $S^1$) of any point $p$ in the fiber over $x\in X$ 
is denoted by $m=m(x)$, as in (\ref{eqn:Zm}).

\begin{definition} For a Seifert bundle $M\rightarrow X$ define the first Chern class as follows. Let $l=\operatorname{lcm}(m(x)\,|\,x\in X)$. 
Denote by $M/l$ the quotient of $M$ by $\ZZ_l\subset S^1$. Then $M/l\rightarrow X$ is a circle bundle with the first Chern class $c_1(M/l)\in H^2(X,\ZZ)$. 
Define
$$c_1(M)={\frac 1l}c_1(M/l)\in H^2(X,\mathbb{Q}).$$
\end{definition}
 
The first orbifold Chern class of the Seifert bundle can be calculated using the formula given by the following result, that we state
only for complex orbifolds.

\begin{proposition}[\cite{M}, Proposition 35]\label{prop:c1} Let $X$ be a cyclic $4$-orbifold with a complex structure 
and $D_i\subset X$ complex curves of $X$ which intersect transversally. 
Let $m_i>1$ such that $\gcd(m_i,m_j)=1$ if $D_i$ and $D_j$ intersect. Suppose that there are given local invariants $(m_i,j_i)$ for each $D_i$ and $j_p$ for every singular point $p$, which are compatible. Choose any $0<b_i<m_i$ such that $j_ib_i\equiv 1 (\operatorname{mod}\,m_i).$   Let ${B}$ be a complex line bundle over $X$. Then there exists a Seifert bundle $M\rightarrow X$ with the given local invariants and the first orbifold  Chern class
  $$c_1(M)=c_1({B})+\sum_i\frac{b_i}{m_i}[D_i].$$
The set of all such Seifert bundles forms a principal homogeneous space under\linebreak $H^2(X,\ZZ)$, where the action corresponds to changing ${B}$.

Moreover, if $X$ is a K\"ahler cyclic orbifold and $c_1(M)=[\omega]$ for the orbifold K\"ahler form, then $M$ is Sasakian.
\end{proposition}

Here $[D_i]$ are understood as cohomology classes in $H^2(X,\mathbb{Q})$.

\begin{proposition}[{\cite[Lemma 34]{M}}]\label{prop:lemma34} 
Assume that we are given a Seifert bundle $M\rightarrow X$ over a cyclic orbifold with the set of isotropy points $P$ and 
branch divisor $\sum (1-{\frac 1{m_i}})D_i$. Let $\mu=\operatorname{lcm}(m_i)$. Then $c_1(M/\mu)=\mu \, c_1(M)$ 
is integral in $H^2(X-P,\ZZ)$.
\end{proposition}  
Also, we need a method of calculating the cohomology of the total spaces of Seifert fibrations, in particular, we need to have a condition which ensures $H_1(M,\mathbb{Z})=0$. Such method is contained in the following theorem.
 Recall that an element $a$ in a free abelian group $A$ is called {\it primitive} if it cannot be represented 
 as $a=kb$ with non-trivial $b\in A$, $k\in\mathbb{N}$, $k>1$. 
\begin{theorem}[\cite{M}]\label{thm:H1(M)} Suppose that $\pi: M\rightarrow X$ is a quasi-regular Seifert bundle over a cyclic orbifold $X$ with isotropy surfaces $D_i$ and set of singular points $P$. Let $\mu=\operatorname{lcm}(m_i)$. Then $H_1(M,\ZZ)=0$ if and only if 
\begin{enumerate}
\item $H_1(X,\ZZ)=0,$
\item $H^2(X,\ZZ)\rightarrow \oplus H^2(D_i,\ZZ_{m_i})$ is onto,
\item $c_1(M/\mu)\in H^2(X-P,\ZZ)$ is primitive.
\end{enumerate}
Moreover, $H_2(M,\ZZ)=\ZZ^k\oplus(\, \mathop{\oplus}\limits_i\ZZ_{m_i}^{2g_i})$, $g_i=\text{genus of}\,\,D_i$, $k+1=b_2(X).$
\end{theorem}
Thus, if one wants to check the assumptions of Theorem \ref{thm:H1(M)}, 
one in particular calculates $H^2(X-P,\ZZ)$ and checks the primitivity of $c_1(M/\mu)$ 
in $H^2(X-P,\ZZ)$. 

\subsection{Fundamental groups of total spaces of Seifert bundles}\label{subsec:fund}
Many questions in topology of Sasakian manifolds are much more difficult to answer under the condition of 1-connectedness. Since the basic tool of the studying Sasakian manifolds is via Seifert fibration $M\rightarrow X$, one needs methods of calculating the fundamental group of $M$. In partuclar we need to check that $\pi_1(M)=1$. We will systematically use the following. 
\begin{definition} The orbifold fundamental group $\pi_1^{\orb}(X)$ is defined as 
   $$
   \pi_1^{\orb}(X)=\pi_1(X-(\Delta\cup P))/\langle \gamma_i^{m_i}=1\rangle,
   $$
where $\langle\gamma_i^{m_i}=1\rangle$ denotes the following relation on $\pi_1(X-(\Delta\cup P))$: for any small loop $\gamma_i$ around a surface $D_i$ in the branch divisor, one has $\gamma_i^{m_i}=1$.
\end{definition}

We will use without further notice the following exact sequence
  $$
  \cdots\rightarrow \pi_1(S^1)=\ZZ\rightarrow\pi_1(M)\rightarrow\pi_1^{\orb}(X)\rightarrow 1.
  $$
It can be found in \cite[Theorem 4.3.18]{BG}. It is easy to see that if $H_1(M,\ZZ)=0$ and $\pi_1^{\orb}(X)$ is abelian, then  
$\pi_1(M)$ must be trivial.  This holds since if $H_1(M,\ZZ)=0$, then $\pi_1(M)$ has no abelian quotients. As $\pi_1^{\orb}(X)$ is assumed abelian,
we find that $\pi_1(M)$ is a quotient of $\ZZ$, hence again abelian. This implies that $\pi_1(M)=0$.

The following result will be used in the calculations of the orbifold fundamental groups.
\begin{proposition}[\cite{N}]\label{prop:nori}  If $Z$ is a smooth simply-connected projective surface with smooth complex curves $C_i$ intersecting transversally and satisfying $C_i^2>0$, then $\pi_1(Z-(C_1\cup \ldots\cup C_r))$
is abelian.
\end{proposition}

\subsection{The sum up of the construction} Thus, to construct a Seifert bundle $M\rightarrow X$ determining a quasi-regular Sasakian structure on $M$ one needs to choose:
\begin{itemize}
\item orbifold data (\ref{eqn:orb})
\item compatible local data $(m_i,j_i)$ and $j_p$ for each singular point $p\in P$, as in Proposition \ref{prop:25-M},
\item any integers $0<b_i<m_i,j_ib_i\cong 1 (\operatorname{mod} m_i)$. 

\end{itemize}
These data will determine a Seifert bundle with $c_1(M)$ given by the formula in Proposition \ref{prop:c1}. Choosing these data in an appropriate way, one constructs quasi-regular Sasakian manifolds. The authors of \cite{CMST},\cite{M},\cite{M1}, \cite{MRT}, \cite{MST1} solve some problems posed in \cite{BG}    choosing  these data in a way to get the required conditions on the second Betti number $b_2(M)$ and on the triviality of $H_1(M,\mathbb{Z})$ using Theorem \ref{thm:H1(M)}. By the method of Subsection \ref{subsec:fund} one derives $\pi_1(M)=0$.

\section{ Existence problems for Smale-Barden manifolds}\label{sec:Smale-Barden}
\subsection{Smale-Barden manifolds} 
A $5$-dimensional simply connected manifold $M$ is called a {\it Smale-Barden manifold}. These manifolds are classified by their second homology group over $\ZZ$ and the so-called {\it Barden invariant} \cite{BG}. 
 In more detail,
let $M$ be a compact smooth oriented simply connected $5$-manifold. 
Let us write $H_2(M,\ZZ)$ as a direct sum of cyclic groups of prime  power order
  \begin{equation*} 
  H_2(M,\ZZ)=\ZZ^k\oplus \big( \mathop{\oplus}\limits_{p,i}\, \ZZ_{p^i}^{c(p^i)}\big),
  \end{equation*}
where $k=b_2(M)$. Choose this decomposition in a way that the second Stiefel-Whitney class map
  $w_2: H_2(M,\ZZ)\rightarrow\ZZ_2$
is zero on all but one summand $\ZZ_{2^j}$. The value of $j$ is unique, it
is denoted by $i(M)$ and is called the Barden invariant. The fundamental question arises, 
{\it which Smale-Barden manifolds admit Sasakian structures?} In the subsequent sections we present results of \cite{CMST}, \cite{MT},\cite{M1},\cite{MRT},\cite{MST1} in the context of this general problem. These articles solve several particular problems posed in \cite{BG}. Each of these questions is discussed separately.  

\subsection{Definite Sasakian structures} 
Recall that the Reeb vector field $\xi$ on a co-oriented contact manifold $(M,\eta)$ determines a $1$-dimensional foliation 
$\mathcal{F}_{\xi}$ called the {\it characteristic foliation}. If we are given a  manifold $M$ with a Sasakian structure $(\eta,\xi,\phi, g)$, then 
one can define  {\it basic Chern classes} $c_k(\mathcal{F}_{\xi})$ of $\mathcal{F}_{\xi}$ which are elements of the 
basic cohomology $H^{2k}_B(\mathcal{F}_{\xi})$ (see \cite[Theorem/Definition 7.5.17]{BG}).     
We say that a Sasakian structure is positive (negative) if $c_1(\mathcal{F}_{\xi})$ can be represented by a positive 
(negative) definite $(1,1)$-form. A Sasakian structure is called null, if $c_1(\mathcal{F}_{\xi})=0$. If none of these, it is called indefinite. 

\subsubsection{Rational homology spheres}
\begin{question}[{\cite[Open Problems 10.3.3 and 10.3.4]{BG}}]\label{quest:bg-main}  
Which simply connected rational homology $5$-spheres admit negative Sasakian structures?
\end{question}

Contrary to this question, there are very few rational homology spheres which admit positive Sasakian structures.

\begin{theorem}[{\cite[Theorem 10.2.19]{BG}, \cite[Theorem 1.4]{K}}]\label{thm:pos-sas} 
Suppose that a rational homology sphere $M$ admits a positive Sasakian structure. 
Then $M$ is spin and $H_2(M,\ZZ)$ is one of the following:
 $$
 0,\,\, \ZZ_m^2,\,\,  \ZZ_5^4,\,\, \ZZ_4^4,\,\,
 \ZZ_3^4,\,\,\ZZ_3^6,\,\, \ZZ_3^8,\,\, \ZZ_2^{2n},\,\, 
 $$
where $n>0$, and $m\geq 2$, $m$ not divisible by $30$. Conversely, all these cases do occur.
\end{theorem}

\begin{theorem}[{\cite[Theorem 10.3.14]{BG}}]\label{thm:pos-sphere} 
Let $M$ be a rational homology sphere. If it admits a Sasakian structure, then it is either positive, 
and the torsion in $H_2(M,\ZZ)$ is restricted by Theorem \ref{thm:pos-sas}, or it is negative. There exist infinitely many such manifolds which admit negative Sasakian structures but no positive Sasakian structures. There exist infinitely many positive rational homology spheres which also admit negative Sasakian structures.
\end{theorem}
One can see that  any torsion group 
which is realizable by a simply connected  Sasakian rational homology sphere $M$ but 
which does not appear in the list given by Theorem \ref{thm:pos-sas} gives an example of an answer to Question \ref{quest:bg-main}.

 The first result on negative Sasakian structures on rational homotopy spheres was obtained by R. Gomez \cite{G}.
\begin{theorem}\label{thm:g} Any simply connected rational homotopy sphere from the list in Theorem \ref{thm:pos-sas} except, possibly, $\mathbb{Z}_m^2,m<5$ and $\mathbb{Z}_2^{2n}$, admits both negative and positive Sasakian structures.
\end{theorem}

In the smaller class of semi-regular Sasakian structures, Question \ref{quest:bg-main} is answered by the following result. 

\begin{theorem}[\cite{MT}]\label{thm:main2} 
Let $m_i\geq 2$ be pairwise coprime, and $g_i={\frac 12}(d_i-1)(d_i-2)$. 
Assume that $\gcd(m_i,d_i)=1$. Let $M$ be a Smale-Barden manifold with $H_2(M,\ZZ)
=\mathop{\oplus}\limits_{i=1}^r\ZZ_{m_i}^{2g_i}$ and spin, with the exceptions
$\ZZ_m^2, \ZZ_2^{2n}, \ZZ_3^6$. Then $M$ admits a negative semi-regular Sasakian structure. 
Conversely, if $M$ is a simply-connected rational homology $5$-sphere admitting a 
semi-regular Sasakian structure, then it must satisfy the above assumptions.
\end{theorem}

Theorem 26 in \cite{CMST} shows that in the quasi-regular case Theorem \ref{thm:main2} does not hold.

The above discussion motivates the first problem we want to address, which
certainly contributes to Question \ref{quest:bg-main}.

\begin{question} \label{quest:bg-main2}
Which simply connected rational homology spheres admit both negative and positive Sasakian structures?
\end{question}

We give a complete answer to Question \ref{quest:bg-main2} in Theorem \ref{thm:neg-pos-main}: all positive Sasakian 
simply connected rational homology spheres also  admit negative Sasakian structures.

Also,  Boyer and Galicki pose the following question.

\begin{question}[{\cite[Open Problem 10.3.6]{BG}}]\label{quest:connected-sum} 
Show that all $\#_k(S^2\times S^3)$ admit negative Sasakian structures 
or determine precisely for which $k$ this holds true.
\end{question}

In Theorem \ref{thm:all_k} we give a complete answer to this question: any $\#_k(S^2\times S^3)$ admits negative Sasakian structures.

\medskip

Let us characterize  a definite Sasakian structure in a different way. Recall that a Sasakian structure is positive (negative), if $c_1(\mathcal{F}_{\xi})$ can be represented by a positive (negative) definite $(1,1)$-form. 
 If $\pi: M\rightarrow X$ is a Seifert bundle determined by a quasi-regular Sasakian structure, then by \cite[Proposition 7.5.23]{BG},  $c_1(\mathcal{F}_{\xi})=\pi^*c_1^{\orb}(X)$. Since $c_1(\mathcal{F}_{\xi})$ is represented by a basic form (see \cite[Section 7.5.2]{BG}), it follows that $c_1(\mathcal{F}_{\xi})$ is represented by a  negative definite $(1,1)$-form if and only if $c_1^{\orb}(X)<0$ (that is, represented by a negative definite orbifold $(1,1)$-form). Thus, we get the following (see also \cite[Section 4]{BGM}).
 
\begin{proposition}\label{prop:c1neg} A $5$-manifold $M$ admits a quasi-regular negative Sasakian structure if and only if the base $X$ of the corresponding Seifert bundle $M\rightarrow X$ has the property that the canonical class  $K_X^{\orb}$ is ample.
\end{proposition}

In the case of rational homology spheres Sasakian structures fall into two classes.  


\begin{proposition}[{\cite[Proposition 7.5.29]{BG}}]\label{prop:can-sphere} Let $M$ be a $(2n+1)$-dimensional rational homology sphere. Any Sasakian structure on $M$  is either positive or negative.
\end{proposition}

\begin{theorem}\label{thm:neg-pos-main} 
Any simply connected rational homology sphere admitting a positive Sasakian structure  admits also a negative Sasakian structure.
\end{theorem}
In order to illustrate the methods we reproduce a part of the  proof from \cite{MST1} without any changes.
By Theorem \ref{thm:g}, it only remains to check that the homology groups
$H_2(M,{\mathbb Z})= {\mathbb Z}_2^{2n}$, $n>0$, and ${\mathbb Z}_m^2$, $m< 5$, 
from Theorem \ref{thm:pos-sas} can be covered. Each of these cases is worked out separately in \cite{MST1}. Here we show the proof only for $\mathbb{Z}^{2n}_2$ (the other case is similar in spirit).

We begin with some known facts on Hirzebruch surfaces \cite{BPV}, \cite{GS}, \cite{H}. 
Fix $n\geq 0$. By definition, the Hirzebruch surface $\mathbb{F}_n$  is the projectivization of the vector bundle
 $$
 \mathcal{O}_{\CP^1}(n)\oplus\mathcal{O}_{\CP^1}.
 $$ 
For a holomorphic section
 $\sigma: \CP^1\rightarrow\mathcal{O}_{\CP^1}(n)$, we
denote by $E_\sigma$ the image of $(\sigma,1): \CP^1\rightarrow \mathcal{O}_{\CP^1}(n)\oplus\mathcal{O}_{\CP^1}$
in $\mathbb{F}_n$. 
The curve $E_0\subset\mathbb{F}_n$ is called the zero section of $\mathbb{F}_n$. As $E_0\equiv E_\sigma$ for all sections $\sigma$,
we have that $E_0^2=n$. Let $C$ denote the fiber of the fibration $\mathbb{F}_n\rightarrow \CP^1$. Then
 $$
 C^2=0,\, E_0\cdot C=1.
 $$
It is known \cite{T} that for $n>0$ the surface $\mathbb{F}_n$ contains a unique irreducible curve $E_{\infty}$ of negative self-intersection:
 \begin{equation*}
 E_{\infty}\cdot E_{\infty}=-n.
 \end{equation*}
This curve is called the section at infinity since it is  given by the image of 
$(\sigma,0): \CP^1\rightarrow \mathcal{O}_{\CP^1}(n)\oplus\mathcal{O}_{\CP^1}$. One can calculate
 $$
  E_{\infty} \equiv E_0-nC, \quad K_{\mathbb{F}_n} \equiv (n-2)C-2E_0,
  $$
where $K_{\mathbb{F}_n}$ is the canonical divisor. Clearly
 $$
 H^2(\mathbb{F}_n,\ZZ)=\ZZ\langle C,E_{\infty}\rangle.
 $$
 
Now we are ready to prove our first main result.

\begin{proposition}\label{prop:1-neg-pos-main} 
The simply connected rational homology sphere with $H_2(M,\ZZ)=\ZZ_2^{2n}$, $n\geq 1$, and spin
admits a negative Sasakian structure.
\end{proposition}

\begin{proof} 
Consider $\mathbb{F}_n$ with the zero section $E_0$, the section at infinity $E_{\infty}$, 
and the fiber $C$ of $\mathbb{F}_n\rightarrow \CP^1$. Let $\beta\geq 1$ and  take a smooth divisor 
 $$
 D\equiv C+\beta E_0.
 $$ 
This exists because $ C+\beta E_0$ is  very ample (it can be represented by a K\"ahler form) and Bertini's theorem is applicable.
 By the genus formula $D^2+K_{\mathbb{F}_n} \cdot D =2g(D)-2$, we get
 \begin{equation}\label{eqn:(2)}
 g(D)={\frac{\beta^2-\beta}2}\cdot n. 
 \end{equation}
One also calculate
 $$
 D\cdot E_{\infty}=(C+\beta E_0)\cdot E_{\infty}= 1,
 $$
as $E_0\cdot E_\infty=0$. Consider also rational curves $D_i=E_{\sigma_i}$, $i=1,\ldots, s$, which are
sections, $D_i\equiv E_0$. These can be taken to intersect transversally (each other and also to $D$).

Let $X$ be the orbifold obtained by the blow-down of $\mathbb{F}_n$ along $E_{\infty}$. 
By Proposition \ref{prop:constr-orbi}, 
$X$ is a cyclic orbifold with a cyclic singularity $p$ of degree $d(p)=n$. 
For simplicity we denote curves on $\mathbb F_n$ and on $X$ by the same letters.
Endow $X$ with an orbifold structure with the branch divisor
 \begin{equation}\label{eqn:(3)}
 \Delta= \left(1-\frac{1}{m}\right) D+ \sum_{i=1}^s \left(1-\frac{1}{m_i}\right) D_i \, ,
  \end{equation}
where $m_1,\ldots, m_s$ are pairwise coprime, and the $m_i$ are coprime to $m$ and $n$ (but possibly
$\gcd(n,m)>1$).

One can easily calculate that
 \begin{align}
 \label{eqn:H_2}
  & H_2(X,\ZZ) = H^2(X-\{p\},\ZZ) = \ZZ\langle C\rangle, \\
  \label{eqn:dual}
  & H_2(X-\{p\},\ZZ) = H^2(X,\ZZ) = \ZZ\langle E_0\rangle.
 \end{align}
Note that $C$ is a divisor passing through the singular point $p$.
Since the (co)homology groups in \eqref{eqn:H_2}, \eqref{eqn:dual} are dual, we can calculate the intersection numbers over $\mathbb{Q}$ and obtain
 $$
 C^2={\frac 1n}
 $$
in $H_2(X,\mathbb{Q})$. 
Indeed,
 we can write $E_0\equiv a\, C$, for some $a\in\ZZ$, and substituting this to $n=E_0^2=a\,E_0\cdot C=a$, we get 
$E_0 \equiv n\,C$. Then $C^2=\frac{1}{n^2} E_0^2=\frac1n$. 
We also have
 $$
 D \equiv C+\beta E_0\equiv C+\beta nC=(1+\beta n)C.
 $$

Considering $X$ as an algebraic variety, we calculate the canonical divisor $K_X$ from the adjunction formula (noting that $C$ is
sing-complex):
 $$
 K_X\cdot C+C^2=-\chi_{\orb}(C) = 2g(C)-2 + \left(1-\frac{1}{n}\right).
 $$
Since $g(C)=0$ we get $K_X\cdot C+{\frac 1n}=-2+(1-{\frac 1n})$.
Writing $K_X=b\, C$, $b\in\mathbb{Q}$, and substituting in the above, we get $b=-(n+2)$. So finally
  $$
  K_X=-(n+2)C.
  $$
Since all $D_i$ are homologous to $E_0$, we obtain the following formula for the orbifold canonical class:
\begin{equation}\label{eqn:(5)}
  \begin{aligned}
  K_X^{\orb}& =K_X+\left(1-{\frac 1m}\right)D+\sum_{i=1}^s\left(1-{\frac 1{m_i}}\right)D_i \\ 
  &= -(n+2)C+\left(1-{\frac 1m}\right)(1+\beta n)C+\sum_{i=1}^s\left(1-{\frac 1{m_i}}\right)nC \\
 &=\left(-(n+2)+\left(1-{\frac 1m}\right)(1+\beta n)+n\sum_{i=1}^s\left(1-{\frac1{m_i}}\right)\right)C. 
 \end{aligned}
 \end{equation}

Now consider a Seifert bundle 
 $$
  M\rightarrow X
  $$
determined by the orbifold data (\ref{eqn:(3)}). We need to choose local invariants $0<j_i<m_i$, $\gcd(j_i,m_i)=1$,
for each $D_i$, and $0<j<m$, 
$\gcd(j,m)=1$, for $D$. These will be fixed later. The local invariant $j_p$ at the singular point is given by
applying Proposition \ref{prop:25-M}.

By Proposition \ref{prop:c1neg}, we know that the existence of a negative 
Sasakian structure on $M$ is equivalent to the ampleness of the orbifold canonical bundle $K_X^{\orb}$, 
and this is equivalent to the positivity of the coefficient in (\ref{eqn:(5)}). Now we specify some coefficients
in order to realize the desired cohomology groups.
Put $m=2$ and $\beta=2$, so by (\ref{eqn:(2)}), we have $g(D)=n$. By Theorem \ref{thm:H1(M)}, this gives $H_2(M,\ZZ)=\ZZ_2^{2n}$.
Calculating the coefficient in (\ref{eqn:(5)}), 
we get
  $$
  -{\frac 32}+n\sum_{i=1}^s \left(1-{\frac 1{m_i}}\right).
  $$
Choosing $m_i$ and $s$ large, we can get positivity for any $n$. As we said, we take $m, m_i$ all of them pairwise coprime, and
also $\gcd(m_i,n)=1$ for all $i$.

Note that we have done this calculation assuming that the assumptions of Theorem \ref{thm:H1(M)} are satisfied. 
The assumption (2) holds for $D_i$ since $H^2(X,\ZZ)=\ZZ\la E_0\ra$, $D_i\equiv E_0$, and $\gcd(m_i,n)=1$.
It holds for $D$ since $D\equiv (1+n\beta)C$, so the map $H^2(X,\ZZ)\to H^2(D,\ZZ_m)$ is $[E_0]\mapsto 1+n\beta$, and
$\gcd(1+n\beta,m)=1$ for $\beta=m=2$.

Let us now check assumption (3) of Theorem  \ref{thm:H1(M)}.
By Proposition \ref{prop:c1},
  $$
  c_1(M)=c_1({B})+{\frac bm}[C+\beta E_0]+\sum_i\frac{b_i}{m_i}[E_0],
  $$
for a line bundle $B$, so $c_1(B)=q[E_0]$, for some integer $q\in \ZZ$.
As we take all $m,m_1,\ldots, m_s$ pairwise coprime, we have 
$\mu=m\cdot m_1\cdots m_s$. 
We need to check that we can arrange for $\mu\,c_1(M)$ to be integral and primitive in $H^2(X-P,\ZZ)$, 
which means that we need
 $$
  c_1(M/\mu)=\mu\,c_1(M)=[C].
 $$
Note also that as $c_1(M)>0$, we have by Proposition \ref{prop:c1} that $M$ is Sasakian.
We compute
 $$ 
 \mu\, c_1(M)=b\cdot m_1\cdots m_s[C+\beta E_0]+\sum_ib_i\cdot m\cdot m_1\cdots\hat{m_i}\cdots m_s[E_0]+\mu\,c_1({B}).
 $$
In $H^2(X,\ZZ)$ we have $[E_0]=n[C]$, so $c_1(M/\mu)=k[C]$ with 
 \begin{equation}\label{eqn:k=1}
 k=(1+{n\beta})bm_1\cdots m_s+n\left(\sum_i b_i mm_1\cdots\hat{m_i}\cdots m_s\right)+ \mu q.
   \end{equation}
Since we already fixed $m=\beta=2$,
we need to arrange $b,b_i,m_i,q$ to get $k=1$. First note that
 \begin{equation}\label{eqn:gcd}
  \gcd((1+{n\beta})m_1\cdots m_s, n mm_1\cdots\hat{m_i}\cdots m_s) =1,
   \end{equation}
by the choice of $m,m_i$. Then there are $\bar b,\bar b_i\in\ZZ$ such that
 $$
 1=(1+{n\beta})\bar b m_1\cdots m_s+n\left(\sum_i \bar b_i mm_1\cdots\hat{m_i}\cdots m_s\right).
 $$
Take $\bar b=b +m q_0$, $\bar b_i=b_i+m_i q_i$, where $0\leq b<m$ and $0\leq b_i<m_i$. Then
 $$
 1=(1+{n\beta}) b m_1\cdots m_s+n\left(\sum_i b_i mm_1\cdots\hat{m_i}\cdots m_s\right) + \left((1+n\beta)q_0+\sum n q_i\right)\mu,
 $$
as required. Finally note that $\gcd(b,m)=1$ and $\gcd(b_i,m_i)=1$, so that we can take local invariants 
$j,j_i$ with $bj\equiv 1 \pmod{m}$ and
$b_ij_i\equiv 1\pmod{m_i}$ and apply Proposition \ref{prop:25-M}.

It remains to show that $\pi_1(M)=1$.  By Proposition \ref{prop:nori}  applied to $\pi_1({\mathbb F}_n-S)$,  $S=D\cup D_1\cup \ldots \cup D_s$, we get that this group is abelian.
Now take the curve $E_\infty$, that intersects transversally $S$ at one point. Take a loop $\delta$ around $E_\infty$, then
$\pi_1({\mathbb F}_n-(S\cup E_\infty))$ is generated by $\delta$ and $\pi_1({\mathbb F}_n-S)$.
Let $\gamma,\gamma_1,\ldots, \gamma_s$ be loops around each of the divisors $D,D_1,\ldots D_s$. 
Take a general fiber $C$ of
${\mathbb F}_n \to \CP^1$, 
and intersecting with $S\cup E_\infty$, we get a homotopy 
$\delta=\gamma^\beta \prod_{i=1}^s\gamma_i$,
since $\beta =C \cdot D$ and $1=C\cdot D_i$, $i=1,\ldots, s$. 
Note also that $\delta,\gamma$ commute since the curves $E_\infty,D$ intersect transversally at one point. 
Also $\gamma_i,\gamma_j$ commute since $D_i,D_j$ intersect transversally (just take the link of the complement of
the two curves around an intersection point, which is a $2$-torus containing both loops; hence they commute in this $T^2$).
Analogously $\gamma_i,\gamma$ commute. Finally, $\delta=\gamma^\beta \prod_{j=1}^s\gamma_j$, so it also
commutes with $\gamma_i$ in $\pi_1({\mathbb F}_n-(S\cup E_\infty))$.

Now recall that $\pi_1^{\orb}(X)$ is the quotient of $\pi_1(X-S)=\pi_1({\mathbb F}_n-(S\cup E_\infty))$ with
the relations $\gamma^m=1$, $\gamma_i^{m_i}=1$. Then we have that
$\pi_1^{\orb}(X)$ is also abelian. It follows that $\pi_1(M)$ is abelian. Since $H_1(M,{\mathbb Z})=0$ we get necessarily $\pi_1(M)=1$. 
By Theorem \ref{thm:H1(M)},  $M$ is a rational homology sphere with $H_2(M,\ZZ)=\ZZ_2^{2n}$

We omit the  case ${\mathbb Z}_m^2$, $m< 5$ refering to \cite{MST1}.
\end{proof}

\begin{remark} In  \cite{PW} a complete classification of Sasaki-Einstein rational homology spheres was obtained. This result is a positive type counterpart to the present work which deals with the negative case. One can find more results on the positive case in \cite{BN}.
\end{remark}

\subsubsection{Quasi-regular negative Sasakian structures on $\#_k(S^2\times S^3)$}\label{sect:quasi-regular-torsionfree}

In this subsection we give a complete answer to  Question \ref{quest:connected-sum}:

\begin{theorem} \label{thm:all_k}
Any $\#_k(S^2\times S^3)$
admits a quasi-regular negative Sasakian structure.
\end{theorem}
The methods of proof of this theorem are close to the proof of Proposition \ref{prop:1-neg-pos-main}, they are based on finding suitable $X$, the pattern of curves and checking the amplness of $K_X^{\text{orb}}$ (see \cite{MST1}).

\subsection{Realization of homological data by Sasakian manifolds in dimension 5}

  It is an important open problem to describe the class of Smale-Barden manifolds which admit Sasakian structures (see \cite[Chapter 10]{BG}). 
Some necessary conditions for a Smale-Barden manifold to carry a Sasakian structure are known.
 
\begin{definition} 
Let $M$ be a Smale-Barden manifold. We say that $M$ satisfies the {\it condition G-K} if the pair 
$(H_2(M,\ZZ),i(M))$ written in the form
 $$
 H_2(M,\ZZ)=\ZZ^k\oplus (\bigoplus_{p,i}\ZZ_{p^i}^{c(p^i)}),
 $$
where $k=b_2(M)$, satisfies all of the following:
\begin{enumerate}
\item $i(M)\in\{0,\infty\}$,
\item for every prime $p$, $t(p):=\#\{i\,| \,c(p^i)>0\}\leq k+1$,
\item if $i(M)=\infty$, then $t(2):=\#\{i\,| \,c(2^i)>0\}\leq k$.
\end{enumerate}
\end{definition}

\begin{question}[{\cite[Question 10.2.1]{BG}}]  
Suppose that a Smale-Barden manifold satisfies the conditon G-K. Does it admit a Sasakian structure?
\end{question}

It is known from \cite[Corollary 10.2.11]{BG} that the condition G-K is necessary for the existence of 
K-contact, and hence Sasakian structures. However, it is not known to what extent it is sufficient. 
K\'ollar \cite{K} have found subtle obstructions to the existence of Sasakian structures on Smale-Barden manifolds with
$k=0$. Also, the recent works  \cite{CMRV} and \cite{MRT} showed the existence of \emph{homology Smale-Barden} 
and true Smale-Barden manifolds which carry a K-contact but do not carry any \emph{semi-regular} Sasakian structure. 
Therefore, finding sufficient conditions is an important problem.

As the condition G-K is a bound on the numbers $t(p)$
controlling the size of the torsion part at primes $p$, it is natural to start by focusing on the case of small values of $t(p)$. Let
 $$
  \bt(X)=\max\{ t(p)|\,p \text{  prime}\}.
  $$
The case of $\bt =0$ is that of the torsion-free Smale-Barden manifolds, where
we only have regular Sasakian structures. The next case 
to analyse is $\bt = 1$. In this direction, we prove the following characterization of Sasakian structures with $\bt=1$.

\begin{theorem}\label{thm:main1} 
Let $M$ be a Smale-Barden manifold whose second integral homology has the form
 \begin{equation}\label{cpi2}
 H_2(M,\ZZ)=\ZZ^k\oplus (\bigoplus_{i=1}^r\ZZ_{m_i}^{2g_i}),
 \end{equation}
and with $i(M)=0,\infty$.
Assume that $k\geq 1$, $m_i\geq 2, g_i\geq 1$, with $m_i$ pairwise coprime, $1\leq i\leq r$. 
Then $M$ admits a semi-regular Sasakian structure.
\end{theorem}

Therefore, all Smale-Barden manifols with $k\geq 1$ and $\bt=1$ admit Sasakian structures, and moreover they admit
semi-regular Sasakian structures.

\subsection{Null Sasakian structures} 
\begin{definition} A Sasakian structure is called {\it null} if $c_1(\mathcal{F}_{\xi})=0$. 
\end{definition} 
It was shown in \cite{BG} that if a Smale-Barden  manifold admits a null Sasakian structure it must be homeomorphic to to the connected sum of at most 21 copies of $S^2\times S^3$. Morever, it is proved that any $M=\#_k(S^2\times S^3)$ with $2\leq k\leq 21$ admits a null Sasakian strcture, except, possibly, $b_1(M)=2$ and $b_2(M)=17$. The following problem was posed in \cite{BG}
\begin{question} (Open problem 10.3.2 in \cite{BG}) Find examples of null Sasakian structures on $\#_2(S^2\times S^3)$ and $\#_{17}(S^2\times S^3)$, or show that none can exist.
\end{question}
The answer is given in \cite{CMST}.
\begin{theorem}[\cite{CMST}, Theorem 1.5] Any $M=\#_k(S^2\times S^3)$ with $2\leq k\leq 21$ admits a null Sasakian structure.
\end{theorem}
The proof of this theorem uses the theory of singular K\"ahler orbifolds. In the context of null Sasakian structures, one needs to construct a cyclic K3-orbifold with the required restrictions on the second Betti number and the approprite Seifert bundle.

\section{ Topology of Sasakian manifolds}\label{sec:sas-top}

 \subsection{Boyer-Galicki program} The following  topological 
obstructions to the existence of Sasakian structures  on a compact manifold $M$ 
of dimension $2n+1$, where known \cite{BG}.
\begin{itemize}
\item vanishing of the odd Stiefel-Whitney classes,

\item the inequality $1\,\leq\, cup(M)\,\leq\, 2n$ on the cup-length,

\item the evenness of the $p^{th}$ Betti number for $p$ odd with $1\, \leq\, p \, \leq\, n$,

\item the estimate on the number of closed integral curves of the Reeb vector field 
(there should be at least $n+1$),

\end{itemize}
 
It is
well-known that the theory of
Sasakian manifolds is, in a sense, parallel to the theory of the K\"ahler manifolds. 
There is a deep theorem of Deligne, Griffiths, Morgan and Sullivan on the
rational homotopy type of K\"ahler manifolds \cite{DGMS}. In the same spirit, rational
homotopical properties of a manifold are related to the existence
of suitable geometric structures on the manifold \cite{FOT}.
Therefore, it is important to build a version of such theory for compact Sasakian
 manifolds.
In \cite[Chapter 7]{BG}, the authors pose the following problems.
\begin{enumerate}
\item Are there obstructions to the existence of Sasakian structures expressed in terms 
of Massey products?

\item There are obstructions to the existence of Sasakian structures expressed in terms 
of Massey products, which depend on basic cohomology classes of the related $K$-contact 
structure. Can one obtain a topological characterization of them?

\item Do there exist simply connected $K$-contact non-Sasakian manifolds (open problem 7.4.1)?

\item Which finitely presented groups can be realized as fundamental groups of compact 
Sasakian manifolds?
\end{enumerate}

\subsection{Rational homotopy properties of Sasakian manifolds}
Following the program formulated by Boyer and Galicki, works \cite{BBMT, BFMT} develop the theory of rational homotopy type of such manifolds.
The well-known theorem of Deligne,Griffiths,Morgan and Sullivan\cite{DGMS}  provides formality as an obstruction to the existence of K\"ahler structure. Also, since triple Massey products obstruct formality, they also obstruct K\"ahler structures. Thus, the  question about homotopic properties related to formality is natural in the Sasakian framework.

\subsubsection{Non-formal and formal Sasakian manifolds}
It was shown in \cite{BFMT} that the triple Massey products do not obstruct Sasakian  structures. The examples below show this.
\begin{example}[Non-formal non-simply connected case,\cite{BFMT}]

Recall from \cite{CFL} that the real Heisenberg group
$H^{2n+1}$ admits a homogeneous regular Sasakian structure with its standard 1-form
$\eta\,=\, dz-\sum_{i=1}^{n} y_{i}dx_{i}$. As a manifold
$H^{2n+1}$ is just ${\mathbb{R}^{2n+1}}$ which can be
realized in terms of $(n+2)\x (n+2)$ nilpotent matrices of the form
\begin{equation}\label{Heisenbergmatrix}
A= \left(
\begin{matrix}1 &a_1 &\cdots &a_n& c \\
0 &1 &0 &\cdots & b_1\\
\vdots && \ddots & \cdots & \vdots \\
0 & \cdots & 0&1 & b_n \\
0 &\cdots &0 & 0& 1
\end{matrix}
\right),
\end{equation}
where $a_i, b_i, c\,\in\, \RR$, $i\,=\,1\, ,\cdots\, , n$. 
Then a global system of
coordinates ${x_i\, , y_i\, , z}$ for $H^{2n+1}$ is defined by
$x_{i}(A)\,=\,a_{i}$, $y_{i}(A)\,=\,b_{i}$, $z(A)\,=\,c$. A standard
calculation shows that we have a basis for the left invariant $1$-forms
on $H^{2n+1}$ which consists of
$$
\{dx_{i}\, ,\, dy_{i}\, ,\, dz-\sum_{i=1}^{n} x_{i}dy_{i}\}\, .
$$

Consider the discrete subgroup $\Gamma$ of $H^{2n+1}$ defined
by the matrices of the form given in \eqref{Heisenbergmatrix} with integer entries.
The quotient manifold
 $$
 M\,:=\,\Gamma\backslash H^{2n+1}
 $$ 
is compact. 
The $1$-forms $dx_i$,
$dy_i$ and $dz-\sum_{i=1}^{n} x_{i}dy_{i}$ descend to $1$-forms $\alpha_i$, $\beta_i$
and $\gamma$ respectively on $M$. We note that
$\{\alpha_i\, , \beta_i\, , \gamma\}$ is a basis for the $1$-forms on $M$. 
Let $\{X_i\, , Y_i\, , Z\}$ be the basis of vector fields on $M$ that is dual to
the basis $\{\alpha_i\, , \beta_i\, , \gamma\}$.
Define the almost contact metric structure 
$(\eta\, , \xi\, , \phi\, , g)$ on $M$ by
 $$
 \eta\,=\, \gamma\, , \quad \xi\,=\,Z\, , \quad \phi(X_i)\,=\,Y_i\, ,\quad
\phi(Y_i)\,=\,-X_i
$$
$$
\phi( \xi)\,=\,0\, , \quad
 g\,=\,\gamma^2 + \sum_{i=1}^{n} ((\alpha_i)^2 + (\beta_i)^2)\, .
$$
Then one can check that $(\eta\,, \xi\,, \phi\, , g)$ is a regular Sasakian structure on $M$.
 Moreover, $M$ is non-formal (this is checked in \cite{BFMT}). 
\end{example}

An example of non-simply connected {\it formal} Sasakian manifold is also constructed in \cite{BFMT}. It is a non-trivial circle bundle over $B^4=(L_3/\Gamma)\times S^1$, where $L_3$ is the simply connected non-nilpotent solvable Lie group and $\Gamma$ is a co-compact lattice. 
 It is checked  that $M^5$ is formal by calculating the model of the constructed Boothby-Wang fibration.

\begin{example}[Non-formal simply connected case]
The series of examples of non-formal simply connected Sasakian manifolds is given by the following result \cite{BFMT}.

\begin{theorem}\label{1-connected-sasak: non-formal}
For every $n\,\geq\, 3$, there exists a simply connected compact regular Sasakian manifold 
$M^{2n+1}$, of dimension $2n+1$, which is non-formal. More precisely, there is a 
non-trivial $3$-sphere bundle over $(S^2)^{n-1}$ which is a non-formal simply connected 
compact regular Sasakian manifold.
\end{theorem}

\end{example}

\subsubsection{Vanishing of higher order Massey products of Sasakian manifolds}
Now we show that the higher Massey products
rule out the possibility of existence of Sasakian structures on a 
compact manifold. 

\begin{theorem}[\cite{BFMT}]\label{prop:higher-massey}
Let $M$ be a compact Sasakian manifold. Then, all the higher order
Massey products for $M$ are zero. 
\end{theorem}
\noindent {\it Sketch of proof}.
Let $(\mathcal{A}\,=\,\bigoplus_{i=0}^{2n}A^i,0)$ be a graded differential algebra with a
non-zero element $\omega\,\in\, A^2$. For every $0\,\leq\, k \,\leq\, n$,
define the Lefschetz map
$$
L_\omega\,:\, A^{n-k} \,\longrightarrow\, A^{n+k}\, ,\ \
\beta\,\longmapsto\, \beta\cdot\omega^{n-k}\, .$$
We say that $\mathcal{A}$ satisfies the hard Lefschetz property if $L_\omega$ is an
isomorphism for every $0 \,\leq\, k \,\leq \,n$. 

\begin{proposition}\label{prop:Leschetz-higher-massey}
Let $(\mathcal{A}\,=\,\bigoplus_{i=0}^{2n} A^i,\, 0)$ be a differential graded commutative
algebra, and let $\omega\,\in\, A^2$ be a nondegenerate element, such that the hard
Lefschetz property with respect to $\omega$ holds. Consider the elementary extension
$(\mathcal{A}\otimes\bigwedge (y), \,d)$ of $(\mathcal{A},\, 0)$, where $\deg(y)\,=\,1$ and
$dy\,=\,\omega$. Then the higher order Massey products $\langle a_1,\cdots,
 a_m\rangle$, where $a_i\,\in\, H^*(\mathcal{A}\otimes \bigwedge (y),\, d)$ and
$m\,\geq\, 4$, all vanish.
\end{proposition}
Having this proposition we continue as follows.
We know that Massey products on a manifold $M$
can be computed by using any model for $M$. Now, let $(M,\eta\, , \xi\, , \phi\, ,g)$ be 
a compact Sasakian manifold of dimension $2n+1$.
Put $\mathcal{A}\,=\,H_{B}^{*}(M)$ the basic cohomology of $M$. From Theorem \ref{Tievsky model} we know that
$(\mathcal{A}\otimes\bigwedge(x),\, D)$, with $|x|\,=\,1$, $D(\mathcal{A})\,=\,0$ and $Dx\,=\,[d\eta]_{B}$,
is a model for $M$.

Since $(M,\eta\, , \xi\, , \phi\, ,g)$ is a Sasakian manifold, 
there is the orthogonal decomposition of the tangent bundle $TM$ of $M$
$$
TM = {\mathcal D} \oplus {\mathcal L}\, ,
$$
where ${\mathcal L}$ is the trivial line subbundle generated by $\xi$, and the transversal
foliation ${\mathcal D}$ is K\"ahler with respect to $\omega\,=\,[d\eta]_{B}\in H_{B}^{2}(M)$.
Then, we can apply El Kacimi's basic hard Lefschetz Theorem for transversally K\"ahler foliations $\mathcal{F}$
on a compact manifold $N$ such that $H_{B}^{cod\,\mathcal{F}}(N)\neq\,0$. (See \cite{EK}
for a more complete statement of the theorem.)
This implies that $(\mathcal{A}, 0)$ satisfies the hard Lefschetz property with respect to $\omega$ and so, by 
Proposition \ref{prop:Leschetz-higher-massey}, all higher Massey products 
of the model $(\mathcal{A}\otimes\bigwedge(x), D)$ are zero. 

\hfill$\square$

\subsection{K-contact non-Sasakian manifolds}
Now we answer the question posed in \cite{BG}: are there {\it simply-connected} K-contact manifolds which do not carry Sasakian structures? Note that this is an analogue of the classical question about the existence of symplectic manifolds with no K\"ahler structure (see \cite{TO}). It is interesting to note that such manifold do exist in each odd dimension. However, the methods of constructing such manifolds are quite different: in dimensions $>7$ the problem is settled by pure rational homotopy theory \cite{BFMT,HT}, in dimension 7 one needs a more refined approach, using rational homotopy theory combined with symplectic surgery \cite{G1} and Donaldson's hyperplane theorem \cite{Auroux}. The result can be found in \cite{MT1}, and we present the proof in order to show the difference.  
\vskip6pt
\subsubsection{ Examples in dimensions dimension $>7$}
\begin{theorem}[\cite{BFMT}]\label{thm:k-contact-ns}
Let $M$ be a simply connected compact symplectic manifold of dimension $2k$
with an integral symplectic form $\omega$. Assume that the
quadruple Massey product in $H^*(M)$ is non-zero. There exists a sphere bundle 
 $$
 S^{2m+1}\,\lrightarrow\, E\,\lrightarrow\, M\, ,
 $$
for $m+1\,>\,k$, such that the total space $E$ is $K$-contact, but $E$ does not
admit any Sasakian structure.
\end{theorem} 
The proof of this theorem is based on Theorem \ref{prop:higher-massey}. One shows that the non-vanishing quadruple Massey products survive in the algebraic model of $E$, which rules out Sasakian structure on $E$. On the other hand, Lerman's theorem \cite{L2} ensures the existence of a K-contact structure.By dimensional reasons we obtain the lower bound for the dimension of the K-contact non-Sasakian manifold $\dim E>7$. An earlier paper \cite{HT} contains the same result proved without using higher order Massey products but based on a cohomological calculation together with the result from \cite{WZ}.

\subsubsection{ Examples in dimension 7}

The dimension 7 deserves some attention, since  the  rational homotopy theory still plays an important role. However, it is not sufficient anymore, and one needs to use symplectic surgery and Donaldson's hyperplane theorem in symplectic geometry (which we do not discuss here refering to \cite{Auroux}, \cite{G1,GS}).

\begin{theorem} \label{thm:main}
There exist $7$-dimensional compact simply connected K-contact manifolds which do not admit a Sasakian structure.
\end{theorem}
Let us sketch the proof (we just shorten the proof of \cite{MT1}).

\begin{proposition} \label{prop:GC}
There exists a simply connected $6$-dimensional symplectic manifold $(M,\omega)$ such that
$\dim \ker (L_\omega: H^2(M)\rightarrow H^4(M))$ is odd.
\end{proposition}
\noindent {\it Sketch of proof of the Proposition}.
Gompf  constructs in \cite[Theorem 7.1]{G1}  an example of a simply connected $6$-dimensional symplectic manifold $(M,\omega)$
which does not satisfy the hard Lefschetz property, that is,
the Lefschetz map $L_\omega: H^2(M)\rightarrow H^4(M)$ is not an isomorphism. 
If $\dim \ker L_\omega$ is already odd then we have finished. 

So let us suppose that $\dim \ker L_\omega$ is even. Take  a 
cohomology class $a\in H^2(M)$ which belongs to the kernel of   $L_\omega$.
In \cite[Lemma 2.4]{C} Cavalcanti proves that given a symplectic manifold $(M,\omega)$ as above satisfying that there
exists a symplectic surface $S\hookrightarrow M$ with $\la a,[S]\ra \neq 0$, then there is another $6$-dimensional 
symplectic manifold $(M',\omega')$ (the symplectic blow-up of $M$ along $S$) satisfying
 $$
 \dim \ker (L_{\omega'}: H^2(M')\rightarrow H^2(M'))= \dim \ker (L_\omega: H^2(M)\rightarrow H^2(M))-1.
 $$ 
The symplectic blow-up construction  of $M$ along $S$ is known \cite{TO}. It is proved that 
the fundamental groups $\pi_1(M')\cong \pi_1(M)$ (here M' denotes the blown-up manifold).

Hence $M'$ is simply connected.
This means that the simply connected $6$-dimensional symplectic manifold $M'$ satisfies that

$\dim \ker (L_{\omega'} : H^2(M')\rightarrow H^4(M'))$ is odd, as required.

It remains to find $S\hookrightarrow M$ as required. The cohomology class $a$ is non-zero, so there
is some $b\in H^4(M,\ZZ)$ such that $a\cup b\neq 0$. It is easy to see that there is a
rank $2$ complex vector bundle $E\to M$ with $c_1(E)=0$, $c_2(E)=2b$. This corresponds
to the fact that the map $[M,B\SU(2)]\to H^4(M,\ZZ)$ given by the second Chern class exhausts
$2\,  H^4(M,\ZZ)$ (we omit the proof of this fact, refering to \cite{MT1}). 

Now take the rank $2$ bundle $E \to M$ just constructed. Assume that $[\omega]$ is a
an integral cohomology class (which can always be done by perturbing $\omega$ slightly
to make it rational and multiplying it by a large integer). Let $L\to M$ be the
line bundle with first Chern class $c_1(L)=[\omega]$. We now use
the asymptotically holomorphic techniques introduced by Donaldson in \cite{Donaldson}.
Specifically, the result of \cite{Auroux} guarantees the existence of a suitable
large $k\gg 0$ and a section of $E\ox L^{\ox k}$ whose zero locus is a symplectic
manifold (an asymptotically holomorphic manifold in fact). This zero locus $S\subset M$
is a symplectic surface, and the cohomology class defined by $S$ is $c_2(E\otimes L^{\ox k})=
c_2(E)+2k c_1(L)=2b+2k[\omega]$.
Therefore $\la a,[S]\ra =\la a, 2b+2k[\omega] \ra=2\la a,b\ra \neq 0$, as required.

\hfill$\square$

We will call the manifold produced in Proposition \ref{prop:GC} the {\it Gompf-Cavalcanti manifold},
because it is constructed by the surgery technique of Gompf \cite{G1} together with the
symplectic blow-up of Cavalcanti \cite{C}. Note however that this is not a unique one but a 
family of manifolds.

\noindent {\it Completion of proof}. We show the existence of simply connected compact K-contact non-Sasakian manifolds in 
dimension $7$ by proving that the Boothby-Wang 
fibration over the Gompf-Cavalcanti manifold is K-contact but non-Sasakian. 

Let $(M,\omega)$ be a Gompf-Cavalcanti manifold as given by Proposition \ref{prop:GC}.
We can assume that $[\omega]$ is an integral cohomology class. Let
 \begin{equation}\label{eqn:M}
  S^1\rightarrow E\rightarrow M
 \end{equation}
be the associated Boothby-Wang fibration.  Now we need to prove that $E$ cannot carry Sasakian structures.

There is an exact sequence 
 $$
 H_2(M)\to H_1(S^1)={\mathbb Z}\to H_1(E) \to 0
 $$ 
from the Serre spectral sequence. 
The map $H_2(M)\to {\mathbb Z}$ is cupping with $[\omega]\in H^2(M)$. Taking $[\omega]$ integral
cohomology class and primitive, we have that $H_2(M)\to {\mathbb Z}$ is surjective and
hence $H_1(E)=0$. Using the  long homotopy 
exact sequence and The Gysin exact sequence one obtains that $\pi_1(E)$ is abelian.
Therefore $E$ is simply connected and

 $$
 b^3(E)=b^3(M)+ \dim (\ker L_\omega :H^2(M)  \to H^4(M)).
 $$

As $M$ is a $6$-manifold, we have that $b^3(M)$ is even (by Poincar\'e duality, the intersection
pairing on $H^3(M)$ is an antisymmetric non-degenerate bilinear form, hence the dimension of
$H^3(M)$ is even). By construction,
$\dim (\ker L_\omega :H^2(M) \to H^4(M))$ is odd, so $b^3(E)$ is odd.
As the third Betti number of a $7$-dimensional Sasakian manifold has to be even (Theorem \ref{rel:betti-basic-numbers}), 
we have that $E$ cannot admit a Sasakian structure.
\vskip6pt
\hfill$\square$

\begin{remark} In the context of works \cite{BBMT,BFMT, MT1} one should also mention \cite{CNY, CNMY} devoted to the analogues of the hard Lefschetz property in the Sasakian framework.
\end{remark}

\subsection{Dimension 5}
This is the most difficult case. The solution to the problem of the existence of simply connected K-contact non-Sasakian manifolds in dimension 5 was a result of a series of works \cite{CMRV, M, M1, MRT} and required further  development of the theory of Seifert fibrations over K\"ahler or sympplectic cyclic 4-orbifolds with various types of singularities. This was described in Section 3.  Therefore, we just formulate the main result obtained by V. Mu\~noz.
\begin{theorem}[\cite{M1}] There exists a Smale-Barden manifold admitting K-contact but not admitting Sasakian structures.
\end{theorem} 
\vskip6pt
\noindent {\bf Acknowledgement. } This article is an extended verison of the talk given at the conference {\it Differential geometric structures and applications} held at the University of  Haifa in May 2023. The author is grateful to the organizers for the pleasant and creative atmosphere.
The author was supported by the National Science Center, Poland, grant no. 2018/31/B/ST1/00053.

\end{document}